\documentclass[a4paper, 11pt, underruledhead]{paper}

\title[Affine polar spaces derived from polar spaces and...]{Affine polar spaces derived from polar spaces 
  and Grassmann structures defined on them}
\author{K. Pra{\.z}mowski, M. {\.Z}ynel}

\usepackage{enumerate, amsmath, amsthm, amssymb, math}

\begin{document}


\def\adjac{\mathrel{\sim}}

\def\toadjac{\mathrel{\lower.35ex\hbox{\baselineskip-3pt\lineskip-3pt\vbox{\hbox{$\sim$}\hbox{$\sim$}}}}}
\let\doadjac\toadjac
\def\soadjac{\lower.25ex\hbox{\scriptsize\baselineskip-2.2pt\lineskip-2.2pt\vbox{\hbox{$\sim$}\hbox{$\sim$}}}}
\def\ssoadjac{\lower.21ex\hbox{\tiny\baselineskip-2.2pt\lineskip-2.2pt\vbox{\hbox{$\sim$}\hbox{$\sim$}}}}
\def\oadjac{\mathchoice{\doadjac}{\toadjac}{\soadjac}{\ssoadjac}}

\def\penc{{\bf p}}
\def\pencx{{\bf p}^\ast}

\def\toadjac{\mathrel{\lower.35ex\hbox{\baselineskip-3pt\lineskip-3pt\vbox{\hbox{$\sim$}\hbox{$\sim$}}}}}
\let\doadjac\toadjac
\def\soadjac{\lower.25ex\hbox{\scriptsize\baselineskip-2.2pt\lineskip-2.2pt\vbox{\hbox{$\sim$}\hbox{$\sim$}}}}
\def\ssoadjac{\lower.21ex\hbox{\tiny\baselineskip-2.2pt\lineskip-2.2pt\vbox{\hbox{$\sim$}\hbox{$\sim$}}}}
\def\badjac{\mathchoice{\doadjac}{\toadjac}{\soadjac}{\ssoadjac}}
\newcommand{\minus}{{\bf -}}

\def\txminus{\lower-.35ex\hbox{\baselineskip4pt\lineskip-7pt\vbox{\hbox{$\minus$}\hbox{$\sim$}}}}
\let\dxminus\txminus
\def\sxminus{\lower-.25ex\hbox{\scriptsize\baselineskip-2pt\lineskip-7pt\vbox{\hbox{$\minus$}\hbox{$\sim$}}}}
\def\ssxminus{\lower.21ex\hbox{\tiny\baselineskip-1.4pt\lineskip-5pt\vbox{\hbox{$\minus$}\hbox{$\sim$}}}}
\def\baminus{\mathchoice{\dxminus}{\txminus}{\sxminus}{\ssxminus}}

\def\K{{\cal K}}
\def\X{{\cal X}}
\def\afhor{{H^\infty}}

\def\propandimprop{\dagger}

\def\colin{{\mbox{\boldmath$L$}}}
\def\colind{{\mbox{\boldmath$L$}^{\propandimprop}}}
\def\TRG{{\mbox{\boldmath$\Delta$}}}

\def\lines{{\cal L}}
\def\izolines{{\cal G}}
\def\peki{{\cal P}}
\def\pekix{{\cal P}^\ast}
\def\pekid{{\cal P}^{\propandimprop}}
\def\subm{\widetilde{\sub}}
\def\upadjac{\mathrel{\adjac^{\mbox{\tiny\boldmath$+$}}}}
\def\downadjac{\mathrel{\adjac_{\mbox{\tiny\boldmath$-$}}}}
\def\paradjac{\mathrel{\adjac^{\shortparallel}}}
\let\botadjac\downadjac
\def\afupadjac{\mathrel{\baminus^{\mbox{\tiny\boldmath$+$}}}}
\def\afdownadjac{\mathrel{\baminus_{\mbox{\tiny\boldmath$-$}}}}
\def\afdownadjacx{\mathrel{\baminus_{\mbox{\tiny\boldmath$-$}}^\ast}}
\let\afbotadjac\afdownadjac
\let\afbotadjacx\afdownadjacx
\def\afupadjac{\mathrel{\baminus^{\mbox{\tiny\boldmath$+$}}}}
\def\afbotadjac{\mathrel{\baminus_{\mbox{\tiny\boldmath$-$}}}}
\def\fixprojdu{\mbox{\boldmath$\goth P$}}
\def\fixprojma{\mbox{\scriptsize\boldmath$\goth P$}}
\def\fixproj{\ensuremath{\mathchoice{\fixprojdu}{\fixprojdu}{\fixprojma}{\fixprojma}}}
\def\fixafdu{\mbox{\boldmath$\goth A$}}
\def\fixafma{\mbox{\scriptsize\boldmath$\goth A$}}
\def\fixaf{\ensuremath{\mathchoice{\fixafdu}{\fixafdu}{\fixafma}{\fixafma}}}
\def\fixpoldu{\mbox{\boldmath$\goth Q$}}
\def\fixpolma{\mbox{\scriptsize\boldmath$\goth Q$}}
\def\fixpol{\ensuremath{\mathchoice{\fixpoldu}{\fixpoldu}{\fixpolma}{\fixpolma}}}
\def\srodek{\mbox{\boldmath$U$}}
\def\pls{partial linear space}

\def\vgen#1{\gen{[#1]}}
\def\agen#1#2{{\mbox{{\boldmath$[$}}{#1}\mbox{{\boldmath$]$}}_{{#2}}}}

\def\starof{{\mathrm{S}}}
\def\polarstar{{\starof_0}}
\def\topof{{\mathrm{T}}}
\def\stars{{\cal S}}
\def\starsx{{\cal S}^\ast}
\def\starsm{{\cal S}^\wedge}
\def\starsw{{\cal S}^\circ}
\def\tops{{\cal T}}
\def\topsx{{\cal T}^\ast}
\def\topsw{{\cal T}^{\circ}}

\def\ProjectiveSpSymb{\mathbf{P}}
\def\PencSpace(#1,#2){\ensuremath{\ProjectiveSpSymb_{#1}(#2)}}
\def\PencSpacex(#1,#2){\ensuremath{\ProjectiveSpSymb^\ast_{#1}(#2)}}
\def\PencSpaced(#1,#2){\ensuremath{\ProjectiveSpSymb^{\propandimprop}_{#1}(#2)}}
\def\GrassmannSpSymb{\mathbf{G}}
\def\GrasSpace(#1,#2){\ensuremath{\GrassmannSpSymb_{#1}(#2)}}
\def\AffineSpSymb{\mathbf{A}}
\def\AfSpace(#1,#2){\ensuremath{\AffineSpSymb_{#1}(#2)}}
\def\KwadrSpSymb{\mathbf{Q}}
\def\KwadrSpace(#1,#2){\ensuremath{\KwadrSpSymb_{#1}(#2)}}
\def\AfPolSpSymb{\mathbf{U}}
\def\AfpolSpace(#1,#2){\ensuremath{\AfPolSpSymb_{#1}(#2)}}
\def\AfpolSpacex(#1,#2){\ensuremath{\AfPolSpSymb^\dagger_{#1}(#2)}}

\def\IsoSpace{{\mbox{\boldmath$\goth Q$}}}
\def\AfPolardu{{\mbox{\boldmath$\goth U$}}}
\def\AfPolarma{{\mbox{\scriptsize\boldmath$\goth U$}}}
\def\AfPolar{\ensuremath{\mathchoice{\AfPolardu}{\AfPolardu}{\AfPolarma}{\AfPolarma}}}
\let\Afpolar\AfPolar
\def\AfPolarx{{\mbox{\boldmath$\goth U$}}^{\propandimprop}}

\def\Quadr{{\mathrm{Q}}}
\def\LineOn(#1,#2){\overline{{#1},{#2}\rule{0em}{1,5ex}}}
\def\horyzont{{\cal H}}

\newenvironment{ctext}{%
  \par
  \smallskip
  \centering
}{%
 \par
 \smallskip
 \csname @endpetrue\endcsname
}

\def\rref#1{\ref{#1})}

\newcounter{sentencex}
\def\thesentencex{\Roman{sentencex}}
\def\labelsentencex{\upshape(\thesentencex)}

\newenvironment{sentencesx}{%
        \list{\labelsentencex}
          {\usecounter{sentencex}\def\makelabel##1{\hss\llap{##1}}
            \topsep3pt\leftmargin0pt\itemindent40pt\labelsep8pt}%
  }{%
    \endlist}

\def\labelsentencey{\bf}

\newenvironment{sentencesy}{%
        \list{\labelsentencey}
          {
            \topsep3pt\leftmargin0pt\itemindent40pt\labelsep8pt}%
  }{%
    \endlist}



\maketitle

\begin{abstract}
  \noindent
  We prove that an affine polar space in the meaning of Cohen and Shult can be
  recovered from one of the three adjacency relations on a Grassmann structure over
  it. 
  The result directly generalizes  the 
  results of our previous work where we use an affine
  space over a vector space equipped with a nondegenerate reflexive form as a
  starting point to the Cohen-Shult affine polar spaces.

  \begin{flushleft}
  \end{flushleft}
\end{abstract}

\section*{Introduction}

Affine polar spaces were first introduced by Cohen and Shult in
\cite{cohenshult}. The idea of that new concept resembles the affine reduct of a
projective space. Start with a polar space  (cf. \cite{Veldkamp},
\cite{BS74}, \cite{Cohen} as well as \cite{Tits} and \cite{Kr91}) and simply
delete a fixed hyperplane in it. 
A hyperplane is understood geometrically, i.e. as a proper subspace with the property
that it meets every line in at least one point.
The authors in \cite{cohenshult} do not only show how to construct an
affine polar space in a polar space but they also provide an axiomatic
characterization of it. In this note we prefer the first framework where we can
refer to the ambient polar space for convenience or clarity. 

There are many analogies between polar spaces and affine polar spaces. Following
Tits (cf. \cite{Tits}), the former can be characterized as gamma spaces (cf.
\cite{Cohen}) which every strong subspace is a projective space, and the latter
are gamma spaces, which every strong subspace is an affine space.

It is one of standard questions posed within framework of geometry of (partial) linear
spaces and Grassmann spaces associated with them, whether adjacency relation on 
$k$-subspaces is sufficient to recover the underlying space.
The main goal of \cite{afpolar} was to show that adjacency relation (actually
there are three 
different adjacency relations) on strong $k$-subspaces
of an affine polar space is sufficient to reconstruct the underlying affine
polar space. 
In that point the paper \cite{afpolar} solves problems analogous to those solved
in \cite{polargras} in the context of geometry of polar spaces.
Note that situation investigated in \cite{polargras} was even simpler a bit,
as there are {\em two} distinct different adjacency relations on strong $k$-subspaces
of a polar space.
In \cite{afpolar} we, however, 
did not follow the idea of Cohen and Shult strictly. Instead we started
from an affine space over a vector space with a reflexive form. 
Consequently, while a polar space (with a few exceptions) can be thought of as
a structure consisting of self-conjugate points and lines of a metric projective 
space, an affine polar space of \cite{afpolar} is a structure of isotropic
(i.e. all) points and isotropic lines of a metric affine space.
Additional result of \cite{afpolar} is
that in terms of aforementioned adjacencies as single primitive  notions we can
also reinterpret the underlying affine space.

A natural question here is what is the gain or what is the loss in
the approach of \cite{afpolar}
versus that of Cohen and Shult. The gain is that 
cases where the index of the affine polar space is 1, e.g. Minkowskian geometry, are included.
The loss is that: the symplectic geometry, 
affine polar spaces that arise by deleting nontangent hyperplanes,
and all affine polar spaces with nondesarguesian planes are excluded. 
So, many Cohen-Shult affine polar spaces were not covered 
by the approach of \cite{afpolar}.

The most important question is whether the main result of \cite{afpolar} can be
carried on to Cohen-Shult affine polar spaces and this paper gives a positive
answer. 
That is
we show that, for each reasonable dimension,
the  underlying Cohen-Shult affine polar space can be recovered from the
binary adjacency on strong $k$-subspaces of it (cf. Theorem~\ref{thm:main}).
We also say that the affine polar space (and other respective structures) are
definable in terms of corresponding adjacencies. The term definable,
interpretable, reinterpretable or recoverable used here may raise some
ambiguity: how to {\em define} points and relations on points in terms of subspaces
and relations on them.
A precise discussion of logical foundations of these problems can be found in 
\cite{Pamb06}, \cite[Sect.~6]{afpolar}.
Without coming into details let us say, roughly, that such a definability means 
that points can be identified with couples of subspaces -- elements of corresponding 
structures, and under this identification suitable relations on points can be 
expressed as relations on corresponding couples of subspaces.
Our reasoning here is standard, like in 
many papers on Chow-like theorems. 
The starting point is to determine and characterize the cliques of all
considered adjacency relations in respective geometry. 
To this point we use
here
the readily apparent relationship between affine polar spaces and
polar spaces like the one between affine and projective  spaces. In consequence
the cliques, as well as strong subspaces, in an affine polar space are reducts
of respective cliques and strong subspaces in the "surrounding" polar space in
which our affine polar space is defined. These cliques are used
to reinterpret  other various notions like ternary collinearity relation or
lines. Then, the critical point is an induction step that brings the dimension
$k$ down by  one to dimension 0, which corresponds to the underlying structure.
In order to do that we simply identify $(k-1)$-subspaces with cliques of type
star and reconstruct our adjacency relations on them.
Such a structure of proof is typical and it was also applied in \cite{polargras}
and \cite{afpolar}.
It can be also applied to investigations on adjacencies of subspaces of a projective
and of an affine space, in particular proving the Chow theorem (cf. \cite{chow}) for corresponding
Grassmann structures (see  \cite{huang3}, \cite{huang4}, see also \cite{huang2}).
It is also worth to mention
that as a direct consequence of a reinterpretability of a sort discussed above 
a suitable theorem in the spirit of Chow follows:
a bijection on $k$-subspaces which preserves (in both directions)
the (currently investigated) adjacency is induced by an automorphism of the underlying space.
So, this theorem remains valid also for affine polar spaces. 
In some cases a theorem like that may follow from general properties of the
adjacency graph (cf. \cite{huang1}).
In our case we have three adjacencies and no general theorem like that of
\cite{huang1} can be used.

Our main result, Theorem~\ref{thm:main}, is formulated as generally as possible, but
it is new only for the case where Cohen-Shult approach misses that of \cite{afpolar}. 
We could have deal with this specific case exclusively here but the arguments are the
same as in general.

\section{Notions, problems}

Let $\fixpol = \struct{\Quadr,\lines}$ be a polar space of index $m$ (cf. \cite{BS74}).
The class of polar spaces is characterized axiomatically, but we do not have to quote
here an adequate axiom system. It suffices to imagine $\fixpol$ as a quadric embedded
into a projective space together with the (projective) lines which lay on it.
Remember, however, that these structures constitute only a part of models.
Traditionally, for points $a, b\in\Quadr$ we write $a\perp b$ when they are
\emph{collinear}. Given two subsets $X, Y\subseteq\Quadr$ we write $X\perp Y$
when $x\perp y$ for all $x\in X$, $y\in Y$. We call a subspace $X$ of $\fixpol$
\emph{strong} if $X\perp X$ (it is called singular in \cite{BS74}).
We write $\sub(\fixpol)$ ($\sub_k(\fixpol)$) for the family of all
(of all $k$-dimensional, respectively) strong subspaces of $\fixpol$.
Clearly, $k \leq m$.
Let $\horyzont$ be a hyperplane of $\fixpol$ (cf. \cite{cohenshult}).
This means that $\horyzont$ is a proper subspace of \fixpol\ such that any line in $\lines$
crosses $\horyzont$. 
Let $L\in\lines$. Note that either $|\horyzont \cap L|\geq 2$ and then $L \subset \horyzont$ 
or $|\horyzont \cap L|=1$.
We write $L^\infty = p$ if $\{  p \} = L\cap\horyzont$.
The point $L^\infty$ is referred to as {\em the improper point} of $L$.

The affine polar space \Afpolar\ is a reduct of \fixpol:
the pointset of \Afpolar\ is the set $\Quadr\setminus\horyzont$
and the lineset $\izolines$ of \Afpolar\ is the set 
\begin{ctext}
  $\{ L\setminus\horyzont\colon L\in\lines \} \setminus\{ \emptyset \}
  = \bigl\{ L\setminus \{ L^\infty \}\colon L\in\lines,\; 
  L\not\subset\horyzont \bigr\}$.
\end{ctext}
Clearly, if $l \in\izolines$ then there is the unique $L\in\lines$
such that $l\subset L$. We see that $L = l \cup\{ L^\infty \}$
and write $\overline{l} = L$ as well as $l^\infty = L^\infty$.
Two lines $l_1, l_2\in\izolines$ are parallel, in symbols
$l_1 \parallel l_2$, iff $l_1^\infty = l_2^\infty$.
Recall that the parallelism $\parallel$
is definable in terms of the incidence structure $\Afpolar$ (cf. \cite{cohenshult}).
The same notation $\perp$, $\sub(\AfPolar)$, $\sub_k(\AfPolar)$ and 
related terminology is used with respect to $\Afpolar$ as to $\fixpol$.

A strong subspace of $\AfPolar$ is 
obtained by deleting $\horyzont$ from a strong subspace of $\fixpol$,
so we can write
\begin{equation}
  \sub_k(\AfPolar) = \left\{ X\setminus \horyzont \colon X\in\sub_k(\fixpol),\;
  X \not\subset \horyzont\right\}.
\end{equation}

Note that if $X\in\sub(\fixpol)$ and $X\not\subset \horyzont$ then 
$\dim_{\fixpol}(X) = \dim_{\Afpolar}(X\setminus\horyzont)$.
In particular, the index of \Afpolar\ (i.e. the maximal dimension of a strong
subspace of \Afpolar) is $m$ as well.

Two properties are crucial:
\begin{enumerate}[(i)]
\item\label{localProj}
  Every $X\in\sub_k(\fixpol)$ carries (as a substructure of \fixpol) the 
  geometry of a $k$-dimensional projective space.
\item\label{localAff}
  Every $A\in\sub_k(\Afpolar)$ carries (as a substructure of \Afpolar)
  the geometry of a $k$-dimensional affine space.
  In particular, it contains a natural parallelism of lines, 
  which is the restriction of the
  parallelism $\parallel$ defined on \Afpolar\ to the set of lines of $A$.
\end{enumerate}
The property \eqref{localProj} can be considered as a characteristic axiom of polar
spaces; similarly, \eqref{localAff} can be considered as a characteristic axiom
of affine polar spaces.
Many of our results concerning classification of 
cliques can be proved synthetically under assumption that the
underlying space is a $\Gamma$-space which satisfies \eqref{localAff}.

For a subspace $A$ of $\Afpolar$ the least subspace of \fixpol\ that contains $A$ 
is written as $\overline{A}$. Then $A = \overline{A}\setminus\horyzont$.
We write $A^\infty := \overline{A}\cap\horyzont$.
Note that if $A\in\sub_k(\Afpolar)$ then $A^\infty\in\sub_{k-1}(\fixpol)$.
It is seen that 
$A^\infty = \left\{ l^\infty\colon l\in\izolines,\;l\subset A \right\}$.
So, we can extend the parallelism of lines to the parallelism of subspaces
$A_1, A_2$, which we write as $A_1\parallel A_2$, by requirement that 
$A_1^\infty = A_2^\infty$.
Note that for 
$\sub(\Afpolar)\ni A_1,A_2 \subset A\in\sub(\Afpolar)$ we have
$A_1^\infty = A_2^\infty$ 
iff $A_1,A_2$ are parallel subspaces of the affine space $A$.

If $X, Y$ are  subspaces of $\fixpol$ we denote by $X\sqcup Y$ the \emph{join}
of $X$ and $Y$ in $\fixpol$, i.e. the least subspace of $\fixpol$ that 
contains $X\cup Y$;
the \emph{meet} of $X$ and $Y$, is simply $X\cap Y$.
For $X, Y\in\sub(\fixpol)$ we have always $X\cap Y\in\sub(\fixpol)$
while $X\sqcup Y\in\sub(\fixpol)$ iff $X\perp Y$.
This paragraph remains true if we replace $\fixpol$ by $\Afpolar$.

Three types of adjacencies on $\sub_k(\AfPolar)$
will be investigated in the sequel:
\begin{eqnarray*}
  A_1 \downadjac A_2 & :\iff & A_1 \cap A_2 \in \sub_{k-1}(\AfPolar),
  \\
  A_1 \upadjac A_2 & :\iff & A_1 \sqcup A_2 \in \sub_{k+1}(\AfPolar),
  \\
  A_1 \badjac A_2 & :\iff & A_1 \downadjac A_2 \Land A_1 \upadjac A_2,
\end{eqnarray*}
These adjacencies may degenerate on $\sub_k(\AfPolar)$:
$\downadjac$ is total for $k=0$ and $\upadjac$ is empty for $k=m$;
in other cases they do not degenerate.
Consequently, 
\begin{ctext}\em
  dealing with $\downadjac$ we assume that $k > 0$, and dealing with $\upadjac$
  we assume that $k < m$. 
\end{ctext}

Let us recall after \cite{polargras} two other adjacencies defined on $\sub_k(\fixpol)$:
\begin{eqnarray*}
  X_1 \afbotadjac X_2 & :\iff & X_1 \cap X_2 \in \sub_{k-1}(\fixpol),
  \\
  X_1 \afupadjac X_2 & :\iff & X_1 \sqcup X_2 \in \sub_{k+1}(\fixpol).
\end{eqnarray*}
From $X_1 \afupadjac X_2$ it follows $X_1 \afbotadjac X_2$, but the converse
implication fails.
Note the following evident but useful relations
($A_1,A_2 \in \sub_k(\AfPolar)$):
\begin{eqnarray}
  \label{zal:down}
  A_1 \downadjac A_2 & \implies & \overline{A_1} \afdownadjac \overline{A_2},
  \\
  \label{zal:up}
  A_1 \upadjac A_2 & \iff & \overline{A_1} \afupadjac \overline{A_2},
  \\
  \overline{A_1}\afbotadjac \overline{A_2} & \iff & 
  A_1 \downadjac A_2 \text{ or } A_1 \parallel A_2. 
\end{eqnarray}
The maximal cliques of the two adjacency relations $\afupadjac$ and
$\afbotadjac$  in the polar space $\fixpol$ has been established in
\cite{polargras}. We give the complete list here as it is needed to get the
cliques in our affine polar space  $\AfPolar$.
\begin{fact}[{\cite[Prop. 3.3, Prop. 3.4]{polargras}}]\label{fct:polkliki}
  Let\/ $\K$ be a subset of\/ $\sub_k(\fixpol)$.
  \leftmargini15pt
  \begin{itemize}
  \item
    $\K$ is a maximal $\afupadjac$-clique iff it has one of the following two forms:
    \begin{enumerate}[\rm(a)]\itemsep-0pt
    \item\label{polklik:typ1}
      $\K = \topof(B) = \left\{ U\in\sub_k(\fixpol)\colon U \subset B \right\}$,
      where $B\in\sub_{k+1}(\fixpol)$, or 
    \item\label{polklik:typ2}
      $\K = [C,M] = \left\{ U\in\sub_k(\fixpol)\colon C \subset U \subset M \right\}$,
      where $C \in\sub_{k-1}(\fixpol)$ and $C\subset M\in\sub_m(\fixpol)$.
    \end{enumerate}
  
  \item
    $\K$ is a maximal $\afbotadjac$-clique iff it has one of the following two forms:
    \begin{enumerate}[\rm(a)]\itemsep-0pt
    \item
      as above, or 
    \setcounter{enumi}{2}
    \item\label{polklik:typ3}
      $\K = \starof(C) = \left\{ U\in\sub_k(\fixpol) \colon C\subset U \right\}$,
      where $C$ is as in \eqref{polklik:typ2}.
    \end{enumerate}
  \end{itemize}    
\end{fact}

In view of \eqref{zal:down} and \eqref{zal:up}, if $\K\subset \sub_k(\AfPolar)$ 
is a maximal $\downadjac$-clique (a maximal $\upadjac$-clique)
then the family 
\begin{ctext}
  $\widetilde{\K} := \bigl\{ \overline{A}\colon A \in\K \bigr\}$
\end{ctext}
is a $\afbotadjac$-clique (a $\afupadjac$-clique resp.),
whose extension (under suitable dimensional assumptions) to a respective 
maximal clique $\overline{\K}$ is unique.
In other words we can simply say that the maximal cliques in $\AfPolar$
are reducts of respective cliques in the polar space $\fixpol$.
With the help of \ref{fct:polkliki} 
we obtain the following list of possible forms
of maximal cliques on $\sub_k(\AfPolar)$.

\begin{prop}\label{prop:afkliki}
  Let\/ $\K$ be a subset of\/ $\sub_k(\AfPolar)$.
  \leftmargini15pt
  \begin{itemize}
  \item
    A maximal $\upadjac$-clique $\K$ may have one of the following forms:
    \begin{enumerate}[\rm(a)]\setcounter{enumi}{3}
    \item\label{afklik:typ1}
      $\K = \topof(B) = \left\{ A\in\sub_k(\AfPolar)\colon A \subset B \right\}$, 
      where $B \in\sub_{k+1}(\AfPolar)$.
      Then 
      $\widetilde{\K} \subset \topof(\overline{B})$.
    \item\label{afklik:typ2}
      $\K = [C,M] = \left\{ A \in \sub_k(\AfPolar)\colon C \subset A \subset M \right\}$,
      where $C \in \sub_{k-1}(\AfPolar)$ and $C \subset M \in\sub_m(\AfPolar)$.
      Then
      $\widetilde{\K} \subset [\overline{C},\overline{M}]$.
    \item\label{afklik:typ3}
      $\K = [A_0,M]^\ast = 
      \left\{ A \in \sub_k(\AfPolar)\colon A_0 \parallel A \subset M \right\}$,
      where $A_0 \in \sub_{k}(\AfPolar)$ and $A_0 \subset M \in\sub_m(\AfPolar)$.
      Then
      $\widetilde{\K} \subset [A_0^\infty,\overline{M}]$.
    \end{enumerate}
    A set of the form \eqref{afklik:typ1}--\eqref{afklik:typ3} is a $\upadjac$-clique;
    in case \eqref{afklik:typ1} it is not maximal iff\/ $k=0$ and $m > 1$,
    while in cases \eqref{afklik:typ2}
    and \eqref{afklik:typ3} it is not maximal iff\/ $0 < k = m-1$.

  \item
    A maximal $\downadjac$-clique $\K$ may have one of the following forms:  
    \begin{enumerate}[\rm(a)]\setcounter{enumi}{6}
    \item\label{afklik:typ4}
      $\K \subset \topof(B)$ with $\topof(B)$ defined as in \eqref{afklik:typ1}
      is a selector of $B^\infty$, i.e. for every $A\in\topof(B)$ there is exactly
      one $A'\in\K$ with $A'\parallel A$.
    \item\label{afklik:typ5}
      $\K = \starof(C) = \left\{ A\in\sub_k(\AfPolar)\colon C\subset A \right\}$,
      where $C\in\sub_{k-1}(\AfPolar)$.
      Then
      $\widetilde{\K} = \starof(\overline{C})$.
    \end{enumerate}
    A set of the form \eqref{afklik:typ4} does not exist when $k=m$, otherwise
    sets of both types are maximal $\downadjac$-cliques.

  \item
    A maximal $\badjac$-clique may have the form \eqref{afklik:typ2} or 
    \eqref{afklik:typ4}.
  \end{itemize}
\end{prop}

\begin{proof}
  Let $\K$ be a maximal $\upadjac$-clique in $\AfPolar$. By \eqref{zal:up}, 
  $\widetilde{\K}$ is a $\afupadjac$-clique. So, in view of \ref{fct:polkliki}
  two possibilities arise: 
  $\widetilde{\K}\subseteq\topof(B)$ for some $B\in\sub_{k+1}(\fixpol)$,
  or
  $\widetilde{\K}\subseteq[C,M]$ for some $C\in\sub_{k-1}(\fixpol)$ and $M\in\sub_m(\fixpol)$
  with $C\subset M$.
  In the first case $B\not\subset\horyzont$ and therefore $\K$ is the maximal  $\upadjac$-clique
  of hyperplanes in the affine space $B\setminus\horyzont$. Thus $\K=\topof(B\setminus\horyzont)$.
  In the second case $M\not\subset\horyzont$, so $M\setminus\horyzont$ is an affine space
  that contains $\K$ and $M\cap\horyzont$ is its horizon. If $C\not\subset\horyzont$, then
  $\K$ is of the form \eqref{afklik:typ2}, otherwise $\K$ is of the form \eqref{afklik:typ3}.
  
  Now, let $\K$ be a maximal $\downadjac$-clique in $\AfPolar$. By \eqref{zal:down},
  $\widetilde{\K}$ is a $\afdownadjac$-clique. Again by  \ref{fct:polkliki} we have
  either 
  $\widetilde{\K}\subseteq\topof(B)$ for some $B\in\sub_{k+1}(\fixpol)$,
  or
  $\widetilde{\K}\subseteq\starof(C)$ for some $C\in\sub_{k-1}(\fixpol)$.
  In the first case $B\not\subset\horyzont$ and $\K$ is a $\downadjac$-clique
  of hyperplanes in the affine space $B\setminus\horyzont$, so no two elements of $\K$
  are parallel. Since $\K$ is maximal there is a hyperplane in $\K$ in every hyperplane 
  direction of $B\setminus\horyzont$. Hence $\K$ has form \eqref{afklik:typ4}.
  In the second case $C\not\subset\horyzont$ as otherwise we would have parallel
  elements in $\K$ which is impossible. Thus $\K\subseteq\starof(C\setminus\horyzont)$.
  Since $\K$ is maximal we get $\K = \starof(C\setminus\horyzont)$.
  
  If $\K$ is a maximal $\badjac$-clique in $\AfPolar$, then it is a clique
  with respect to both $\upadjac$ and $\downadjac$, so it is of the
  form \eqref{afklik:typ2} or \eqref{afklik:typ4}.
  
  If a subset $\K$ of $\sub_k(\AfPolar)$ is of one the forms 
  \eqref{afklik:typ1} - \eqref{afklik:typ5},
  then it is evidently a clique of a corresponding adjacency. It is maximal
  iff $\widetilde{\K}$ is maximal.
\end{proof}

We denote classes of the sets of the form ``$\K = \ldots$'' 
introduced above as follows:
\begin{ctext}
  \eqref{afklik:typ1} -- $\tops$,\qquad
  \eqref{afklik:typ2} -- $\starsm$,\qquad
  \eqref{afklik:typ3} -- $\starsx$,\qquad
  \eqref{afklik:typ4} -- $\topsx$,\qquad
  \eqref{afklik:typ5} -- $\stars$.
\end{ctext}

A polar space $\fixpol$ determines a partial linear space 
$\PencSpace(k,\fixpol)$ 
(cf. \cite{Cohen}, \cite{polargras})
with the point set $\sub_k(\fixpol)$ and with the
line set $\peki_k(\fixpol)$ consisting of the $k$-pencils, i.e. with the sets
\begin{ctext}
  $\penc(C,B) = \left\{ X\in\sub_k(\fixpol)\colon C \subset X \subset B \right\}$,
  \\
  where $C\in\sub_{k-1}(\fixpol)$ and $C\subset B\in\sub_{k+1}(\fixpol)$.
\end{ctext}
Accordingly, we define on the set $\sub_k(\AfPolar)$ two structures of 
a partial linear space:
\begin{ctext}
  $\PencSpace(k,\AfPolar) = \struct{\sub_k(\AfPolar),\peki_k(\AfPolar)}$ 
  \quad and \quad 
  $\PencSpaced(k,\AfPolar) = \struct{\sub_k(\AfPolar),\pekid_k(\AfPolar)}$,
\end{ctext}
where
$\peki_k(\AfPolar)$ consists of the pencils $\penc(C,B)$ with $C\in\sub_{k-1}(\AfPolar)$
and $C \subset B \in\sub_{k+1}(\AfPolar)$, defined analogously as pencils over $\fixpol$,
and
$\pekid_k(\AfPolar)$ consists of the nonvoid sets
$p\vert_{-\horyzont} = \left\{ X\setminus \horyzont \colon X \in p, \; 
X\not\subset \horyzont \right\}$
with $p\in\peki_k(\fixpol)$.
Note that for 
$C\in\sub_{k-1}(\AfPolar)$ and $C \subset B \in\sub_{k+1}(\AfPolar)$ 
we have 
$\penc(C,B) = \penc(\overline{C},\overline{B})\vert_{-\horyzont}$
and 
$\widetilde{\penc(C,B)} = \penc(\overline{C},\overline{B})$.
Consequently,
$\peki_k(\AfPolar)\subset\pekid_k(\AfPolar)$.
Let $A_0\in\sub_k(\AfPolar)$ and $A_0 \subset B\in\sub_{k+1}(\AfPolar)$;
we set 
$\pencx(A_0,B) = \left\{ A\in\sub_k(\AfPolar)\colon A_0\parallel A \subset B \right\}$.
Let $\pekix_k(\AfPolar)$ be the class of all the sets of the form $\pencx(A_0,B)$.
It is seen that 
$\pencx(A_0,B) = \penc(A_0^\infty,\overline{B})\vert_{-\horyzont}$ 
and
$\widetilde{\pencx(A_0,B)} = \penc(A_0^\infty,\overline{B}) \setminus\{B^\infty\}$.
Finally, we have
$\pekid_k(\AfPolar) = \peki_k(\AfPolar) \cup \pekix_k(\AfPolar)$.
Structures $\PencSpace(k,\Afpolar)$ and $\PencSpaced(k,\Afpolar)$
are usually referred to as {\em spaces of $k$-pencils} or \emph{Grassmann spaces over} $\Afpolar$.
In terms of these structures we can say that
\begin{ctext}\em
  $\upadjac$ on $\sub_k(\AfPolar)$ is the binary collinearity in
  $\PencSpaced(k,\AfPolar)$, and
  \\
  $\badjac$ on $\sub_k(\AfPolar)$ is the binary collinearity in $\PencSpace(k,\AfPolar)$.
\end{ctext}
Notice that the above definitions of pencils require that $k < m$ and we will assume that
implicitly when referring to either $\PencSpace(k,\Afpolar)$ or $\PencSpaced(k,\Afpolar)$.
For $k=0$ we have $\Afpolar \cong \PencSpace(0,\Afpolar) = \PencSpaced(0,\Afpolar)$.

One could note a similarity of the above construction to the, analogous, construction
of the space of pencils $\PencSpace(k,{\goth A})$ associated with an affine space $\goth A$
(cf. definitions in \cite{tallini}, \cite{bichara}, or, in a more modern paper \cite{dentice}).

Let us point out an evident but useful consequence of \eqref{localAff}.

\begin{fact}\label{localAffpek}
  Let $D\in\sub(\Afpolar)$, $\dim(D) > k$.
  Then the restriction of\/
  $\PencSpace(k,\Afpolar)$ to the segment 
  $[\emptyset,D] = \{A\in\sub(\Afpolar)\colon A\subset D\}$
  is $\PencSpace(k,D)$, and the restriction of\/ $\PencSpaced(k,\Afpolar)$
  is $\PencSpaced(k,D)$.
\end{fact}

One can say that `locally' $\PencSpace(k,\Afpolar)$ is an affine Grassmann space.

One more fact, which is a consequence (simple: verify axioms of a polar space) 
and an easy (formal) strengthening of \cite[Theorem~3.5]{polargras} is needed.
\begin{fact}\label{localPolar}
  Let $C\in\sub(\fixpol)$, $\dim(C) = k-1 < m$.
  Then the restriction of $\PencSpace(k,\fixpol)$ to the segment
  $[C,\Quadr] = \{ X\in\sub(\fixpol)\colon C\subset X \}$
  in $\PencSpace(k,\fixpol)$ is a polar space.
\end{fact}

\medskip
Our goal is as follows.
\begin{thm}\label{thm:main}
  If\/ $0 < k \leq m$ then for $\adjac = \downadjac$, while 
  if\/ $k \leq m - 1$ then for $\adjac \;\in \{ \upadjac, \badjac \}$
  the affine polar space $\AfPolar$ can be recovered from 
  $\struct{\sub_k(\AfPolar), \adjac}$.
\end{thm}

The case where $m = 1$ is excluded in the Cohen-Shult axiom system in
\cite{cohenshult}  but it is covered by both constructive approaches: the one in
\cite{afpolar} and  the other in \cite{cohenshult} consisting in hyperplane
removal. 

Note that $\PencSpace(0,\Afpolar) = \PencSpaced(0,\Afpolar)$ is,
up to an isomorphism, the affine polar space $\Afpolar$ itself.
Hence the following is immediate by \ref{thm:main}.

\begin{cor}\label{cor:main}
  $\AfPolar$ is definable in both\/ $\PencSpace(k,\Afpolar)$
  and\/ $\PencSpaced(k,\Afpolar)$, provided that $k < m$.
\end{cor}

\section{The reasoning}

Provided that $k < m$, 
let us write $\colin(A_1,A_2,A_3)$ if $A_1,A_2,A_3$ are
collinear in  $\PencSpace(k,\Afpolar)$, and $\colind(A_1,A_2,A_3)$ if they are
collinear in $\PencSpaced(k,\Afpolar)$. 

We also introduce an auxiliary structure 
\begin{ctext}
  $\GrasSpace(k,\Afpolar) := \struct{\sub_k(\Afpolar),\sub_{k+1}(\Afpolar),\subset}$,
\end{ctext}
which is a partial linear space sometimes called {\em a Grassmannian over} $\Afpolar$.
Note that 
$\GrasSpace(0,\Afpolar)\cong\Afpolar\cong\PencSpace(0,\Afpolar) = \PencSpaced(0,\Afpolar)$.
Later, we will use the first isomorphism to show that $\Afpolar$ can be
reconstructed in terms of adjacency on $k$-subspaces.

\begin{fact}\label{fact:selectors0}
  Let $\K\in\topsx$.
  Then $|\bigcap\K|\leq 1$.
\end{fact}
\begin{proof}
  Let $\K\subset\topof(B)$ for $B$ as in \ref{prop:afkliki}\eqref{afklik:typ4}.
  Clearly, $\bigcap\K$ is a (affine) subspace of $B$.
  Suppose that $\bigcap\K$ contains a line $L$ of $\Afpolar$. Then $B^\infty$ contains 
  a hyperplane $C$ that misses $L^\infty$. On the other hand there is $A\in\K$
  with $A^\infty = C$. From assumption $L\subset A$ and a contradiction arises.
\end{proof}

\begin{fact}\label{fct:geoonkliks}
  Let $\K_1\in\tops$, $\K_2\in\starsm$, $\K_3\in\starsx$,
  $\K_4\in\topsx$, and $\K_5\in\stars$.
  \begin{sentences}\itemsep-2pt
  \item\label{fct:geoonkliks:1}
    Either $|\K_1 \cap \K_2|\leq 1$ or
    $\K_1 \cap \K_2\in\peki_k(\AfPolar)$.
  \item\label{fct:geoonkliks:2}
    Either $|\K_1 \cap \K_3|\leq 1$ or
    $\K_1 \cap \K_3\in\pekid_k(\AfPolar)$.
  \item\label{fct:geoonkliks:3}
    $|\K_2 \cap \K_3|\leq 1$. 
  \item\label{fct:geoonkliks:4}
    $\K_1$ is a strong subspace of\/ $\PencSpaced(k,\AfPolar)$; it carries geometry of
    a $(k+1)$-di\-men\-sio\-nal dual affine space i.e. of a 
    $(k+1)$-dimensional projective space with one point deleted.
    It is a subspace of\/ $\PencSpace(k,\AfPolar)$ but not strong.
  \item\label{fct:geoonkliks:5}
    $\K_2$ is a strong subspace both in $\PencSpaced(k,\AfPolar)$
    and in $\PencSpace(k,\AfPolar)$;
    it carries geometry of a $(m-k)$-dimensional projective space.
  \item\label{fct:geoonkliks:6}
    $\K_3$ is a strong subspace in $\PencSpaced(k,\AfPolar)$ and 
    an anti-clique in $\PencSpace(k,\AfPolar)$;
    it carries geometry of a $(m-k)$-dimensional affine space.
  \item\label{fct:geoonkliks:7}
    If\/ $\bigcap\K_4\neq\emptyset$, 
    then $\K_4$ carries geometry of a projective space. In general, however
    the geometry of $\K_4$ is much more complex. 
  \item\label{fct:geoonkliks:8}
    If\/ $k < m$ then $\K_5$ is a subspace in $\PencSpaced(k,\AfPolar)$
    but it is not strong;  the restriction of\/ $\PencSpace(k,\Afpolar)$
    to $\K_5$ carries geometry  of a polar space.
  \end{sentences}
  Every line of\/ $\PencSpace(k,\Afpolar)$ has form $\K_1 \cap \K_2$
  for some $\K_1\in\tops$ and $\K_2\in\starsm$;
  every line of\/ $\PencSpaced(k,\Afpolar)$ has form $\K_1 \cap \K_2$
  for some $\K_1\in\tops$ and $\K_2\in\starsm\cup\starsx$.
\end{fact}

\begin{proof}
  The reasoning follows by close inspection of \ref{prop:afkliki} so, we use its notation.
  
  \eqref{fct:geoonkliks:1}:
  If $|\K_1 \cap \K_2|>1$, then $\K_1 \cap \K_2 = [C, B] = \penc(C,B)$, where $C, B$ are as 
  in \ref{prop:afkliki} \eqref{afklik:typ1}, \eqref{afklik:typ2} and $C\subset B$.
  
  \eqref{fct:geoonkliks:2}:
  If $|\K_1 \cap \K_3|>1$, then $\K_1 \cap \K_3 = [A_0^\infty, B]^\ast = \pencx(A_0,B)$, where $A_0, B$ are as 
  in \ref{prop:afkliki} \eqref{afklik:typ1}, \eqref{afklik:typ3} and $A_0\subset B$.
  
  \eqref{fct:geoonkliks:3}: 
  Immediate by \ref{prop:afkliki} as no two elements of $\K_2$ are parallel.

  In cases 
  \eqref{fct:geoonkliks:4},
  \eqref{fct:geoonkliks:5},
  \eqref{fct:geoonkliks:6}, and
  \eqref{fct:geoonkliks:7}
  the corresponding $\K_i$ consists of subspaces of an affine space:
  of $B$, $M$, $M$, and $B$ in respective cases.
  Our claim is a consequence of \ref{localAffpek}
  and known properties of affine Grassmann spaces.
  In \eqref{fct:geoonkliks:7} we use, additionally, \ref{fact:selectors0},
  as in this case 
  $\K_4 = [a,B] = \{ A\in\sub_k(\Afpolar)\colon a \in A \subset B \}$
  provided that a point $a$ is in $\bigcap\K_4$.
  
  \eqref{fct:geoonkliks:8}:
  It is evident that the map
  $\sub(\Afpolar)\ni X\mapsto\overline{X}\in\sub(\fixpol)$
  establishes an isomorphism of $\K_5$ 
  (more precisely: of the restriction of $\PencSpace(k,\Afpolar)$ to $\K_5$)
  onto the segment 
  $\widetilde{\K_5} = [\overline{C},Q]$ in $\PencSpace(k,\fixpol)$.
  Our claim is a consequence of \ref{localPolar}.
\end{proof}

Let us point out some deviations for extreme values of $k$ and $m$:
\begin{sentences}\itemsep-2pt
\item
  If $m=1$ then $\badjac$ and $\upadjac$ are sensible only for $k=0$ and
  then $\topof(B)\in\tops$ is a line of $\Afpolar$.
\item
  If $k = m$ then the classes
  $\tops$, $\starsm$, $\starsx$, $\topsx$ are void.
\item
  If $k = m-1$ then $[C,M]\in\starsm$ is a line of $\PencSpace(k,\Afpolar)$
  and $[A_0,M]^\ast\in\starsx$ is a line of $\PencSpaced(k,\Afpolar)$.
  In that case a line of $\PencSpace(k,\Afpolar)$ has exactly one extension to a
  maximal $\upadjac$-clique.
\item
  The geometries on the elements of $\tops$, $\starsm$, and $\starsx$ are
  pairwise distinct provided that $k\neq 0$ or $k\neq m-1$.
  If $k=0$ and $m=1$ then $\tops$ and $\starsx$ both consist of affine lines.
\end{sentences}

\begin{lem}\label{lem:linesx}
  Let $A_1,A_2\in\sub_k(\AfPolar)$ with $A_1\upadjac A_2$. Set $B := A_1 \sqcup A_2$ and
  \begin{equation}
    \X := \bigcap\bigl\{ \K\in\tops\cup\starsm\cup\starsx
    \colon A_1,A_2\in\K \bigr\}.
  \end{equation}
  If $A_1 \downadjac A_2$, then $\X = \penc(A_1\cap A_2, B)$
  and if $A_1 \parallel A_2$, then $\X = \pencx(A_1,B)$.
  In any case $\X\in\pekid_k(\Afpolar)$.
  Moreover,
  \begin{equation}\label{eq:linesx:X}
    \X = \bigcap\bigl\{ \K\in \starsm\cup\starsx
    \colon A_1,A_2\in\K \bigr\}.
  \end{equation}

  \par
  Let $k < m-1$. Then
  \begin{equation}\label{eq:linesx:def}
    A_3 \in\X\ (\text{i.e. } \colind(A_1,A_2,A_3)) 
    \iff (\forall A)      
      [\; A\upadjac A_1,A_2 \implies A\upadjac A_3\; ]
  \end{equation}
  for arbitrary $A_3\in\sub_k(\Afpolar)$.
\end{lem}

\begin{proof}
  In view of \ref{prop:afkliki} and \ref{fct:geoonkliks}\eqref{fct:geoonkliks:3}
  under our assumptions $A_1, A_2$ are in two of the three possible maximal
  $\upadjac$-cliques: one of type $\tops$ and the other of type $\starsm$  or
  $\starsx$, depending on whether $A_1 \downadjac A_2$ or $A_1 \parallel A_2$
  respectively. The first statement follows from
  \ref{fct:geoonkliks}\eqref{fct:geoonkliks:1}-\eqref{fct:geoonkliks:3}. 

  To prove \eqref{eq:linesx:X} set $C:=\overline{A_1}\cap\overline{A_2}$ and observe that 
    $$\bigcap\bigl\{\widetilde{\K}\colon A_1,A_2\in\K\in\starsm\cup\starsx \bigr\} = 
        \bigcap\Bigl\{\bigl[C,M\bigr]\colon M\in\sub_m(\fixpol), \overline{A_1}\cup\overline{A_2}\subset M\Bigr\} = 
        [C,\overline{B}].$$
  Deleting $\horyzont$ from $\fixpol$ that set becomes either 
  $\penc(A_1\cap A_2, B)$, or $\pencx(A_1,B)$ depending on whether
  $C\not\subset\horyzont$ or $C\subset\horyzont$ respectively.

  To prove \eqref{eq:linesx:def} note that the right hand side of it 
  means that $A_3$ is $\upadjac$-adjacent to every element of all the cliques
  $A_1, A_2$ belong to. 
  In particular, $A_3$ belongs to each of the maximal $\upadjac$-cliques that
  contains $A_1,A_2$.
  This suffices to state that equivalently $A_3$ is
  in the meet of the two appropriate cliques, i.e. $A_3$ is
  collinear with $A_1, A_2$ in  $\PencSpaced(k,\Afpolar)$.
\end{proof}

\begin{cor}\label{cor:ua->pencp}
  Let $k < m -1$. Then\/ $\PencSpaced(k,\Afpolar)$ is definable in terms of 
  $\upadjac$ on $\sub_k(\Afpolar)$.
  If\/ $0 < k = m-1$, then the structure $\GrasSpace(k,\Afpolar)$
  is definable in terms of $\upadjac$.
\end{cor}

\begin{cor}\label{cor:difcliq-ua}
  If\/ $k < m-1$, then the
  three types $\tops$, $\starsm$, and $\starsx$ of maximal 
  $\upadjac$-cliques are distinguishable in terms of $\upadjac$,
  as they (as subspaces of $\PencSpaced(k,\Afpolar)$) carry distinct geometries.
\end{cor}

\begin{cor}\label{cor:pencp->penc}
  The class of maximal strong subspaces of\/ $\PencSpaced(k,\Afpolar)$ is equal to 
  $\tops \cup \starsm \cup \starsx$ when $k < m-1$, and it is equal to $\tops$
  when $k = m-1$.
  The classes $\pekix_k(\Afpolar)$ and $\peki_k(\Afpolar)$ are distinguishable
  in terms of geometry of\/ $\PencSpaced(k,\Afpolar)$, provided $k\neq 0$;
  consequently, $\PencSpace(k,\AfPolar)$ is definable in $\PencSpaced(k,\Afpolar)$
\end{cor}
\begin{proof}
  The first claim follows immediately from \ref{fct:geoonkliks}.
  To prove the second claim it suffices to note that if $p\in\pekid_k(\Afpolar)$  then
  \begin{itemize}\def\labelitemi{}\itemsep-2pt
  \item
    $p\in\pekix_k(\AfPolar)$ iff  $p\subset \K$ for some $\K\in\starsx$
    and 
  \item
    $p\in\peki_k(\AfPolar)$ iff $p\subset \K$ for some $\K\in\starsm$.
  \end{itemize}
  This closes the proof in case $k < m-1$.
  Let $k = m-1$, then ${\cal X} = \topof(B)\in\tops$ is simply 
  a dual affine space, i.e. a projective space with one point deleted;
  it is known that ``affine'' lines on $\cal X$ 
  (which are exactly the elements of $\pekix_k(\Afpolar)$ contained in $\cal X$)
  can be distinguished from the 
  class of all the lines on $\X$ (cf. \cite{sapls}).
\end{proof}

  Note that the two types of pencils within $\PencSpaced(k,\Afpolar)$ 
  can be directly distinguished as follows. Let $p\in\pekid_k(\AfPolar)$. 
  We have $p\in\pekix_k(\AfPolar)$ iff there is 
  a triangle $\Delta$ in $\PencSpaced(k,\Afpolar)$ such that $p$ misses 
  the vertices of $\Delta$ and crosses exactly two of its sides.

\begin{lem}\label{lem:difcliq-da-ba}
  Let $\K\in\topsx\cup\stars$ and $A \in\K$. 
  \begin{sentences}\itemsep-2pt
  \item 
    Let $\K\in\stars$. 
    If $k=m$ we assume, additionally, that $\fixpol$ is not of type $\sf D$ 
    (which is read, in our terminology, as 
    $|\starof(C)|\geq 3$ with $C\in\sub_{k-1}(\fixpol)$).
    Then an element of\/ $\topsx\cup\stars$ that
    contains all the elements of\/ $\K$ except, possibly, $A$
    contains $A$ as well, and thus 
    it coincides with $\cal K$.
  \item
    Let $\K\in\topsx$ (note that then $k<m$).
    Then there is $A'\in\sub_k(\Afpolar)$ such that $A\neq A'$
    and
    $\K\setminus\{ A \} \cup \{ A' \}\in\topsx\cup\stars$.
  \end{sentences}
  Consequently, the two types $\topsx$ and $\stars$ of $\downadjac$-cliques
  are distinguishable in terms of $\downadjac$.
 \par
  The same property distinguishes $\topsx$ from $\starsm$ and thus
  these two types of $\badjac$-cliques
  are distinguishable in terms of $\badjac$.
\end{lem}

\begin{proof}
  Recall by \ref{prop:afkliki}\eqref{afklik:typ4} that a clique $\K$ in $\topsx$ is
  a selector of $B^\infty$, i.e. the set of distinct representatives from all possible
  directions of $k$-subspaces in $B$. Every such a representative $A$ can be
  selected up to parallelism, in other words every $A$ can be replaced by $A'$
  such that $A'\neq A$, $A'\parallel A$ and $A'\in\topof(B)$. This is not doable
  with elements in cliques of type $\stars$ or $\starsm$.
\end{proof}

\begin{lem}\label{lem:colin}
  Let $A_1,A_2,A_3 \in\sub_k(\Afpolar)$ be pairwise distinct.
  If\/ $k < m$, then we have
  \begin{equation}\label{eq:colindownadjac}
    \colin(A_1,A_2,A_3) \iff (\exists \K_1\in\topsx)(\exists \K_2\in\stars)
      [\; A_1,A_2,A_3 \in \K_1,\K_2\; ]
  \end{equation}
  and
  \begin{equation}\label{eq:colinbadjac}
    \colin(A_1,A_2,A_3) \iff (\exists \K_1\in\topsx)(\exists \K_2\in\starsm)
      [\; A_1,A_2,A_3 \in \K_1,\K_2\; ].
  \end{equation}
\end{lem}

\begin{proof}
  \ltor
  Assume that $A_1,A_2,A_3\in\penc(C, B)$, where 
  $C\in\sub_{k-1}(\AfPolar)$ and $C\subset B\in\sub_{k+1}(\AfPolar)$.
  As no two of $A_1,A_2,A_3$ are parallel take a selector $\K_1\in\topsx$
  of $B^\infty$ with $A_1,A_2,A_3\in\K_1$, and take $\K_2:=\starof(C)$
  to get \eqref{eq:colindownadjac} or $\K_2:=[C,M]$ for some $M\in\sub_m(\AfPolar)$
  with $B\subset M$ to get \eqref{eq:colinbadjac}. We are through here by \ref{prop:afkliki}.
  
  \rtol
  Assume that $\K_1\subset\topof(B)$ and $\K_2=\starof(C)$  
  or $\K_2=[C,M]$, for some $B, C, M$ like in \ref{prop:afkliki}.
  In both cases $\K_1\cap\K_2\subseteq\penc(C,B)$ and thus $A_1,A_2,A_3$
  are collinear in $\PencSpace(k,\AfPolar)$.
\end{proof}

\begin{cor}\label{cor:da+ba->penc}
  If\/ $k < m$, then the structure $\PencSpace(k,\AfPolar)$ is definable both in
  terms of  $\downadjac$ and in terms of $\badjac$.
\end{cor}

\begin{fact}\label{fact:selectors}
  Let $\K\in\topsx$ and take the least subspace $\K'$ of\/ 
  $\PencSpace(k,\Afpolar)$ that contains $\K$.
  If \ $\bigcap\K\neq \emptyset$ then $\K' = \K$.
  If \ $\bigcap\K = \emptyset$ then
  $\K' = \topof(B)$ for some $B\in\sub_{k+1}(\Afpolar)$.
\end{fact}

\begin{proof}
  Let $\K\subset\topof(B)$ for $B$ as in \ref{prop:afkliki}\eqref{afklik:typ4}.
  From \ref{fact:selectors0} there are two cases to consider.
  Firstly, let $a$ be a common point of all $A\in\K$.  Then 
  $\K = [a,B]$,
  which is, already, a subspace of $\PencSpace(k,\Afpolar)$.
  If $\bigcap\K = \emptyset$, the claim is evident.
\end{proof}

Note that in the first case of \ref{fact:selectors} $\K'$ ($=\K$ here) carries
geometry of some projective space and in the second case the geometry of a dual
affine  space. 
\begin{cor}\label{cor:mocnesubWpenc}
  A maximal strong subspace of\/ $\PencSpace(k,\Afpolar)$ is either an element of 
  $\starsm$ or it has form $[a,B]$ with $B\in\sub_{k+1}(\AfPolar)$ and a point
  $a$ of\/ $\Afpolar$ on $B$.
\end{cor}
In particular case where $k = 1$, $m = 2$ maximal strong subspaces of 
$\PencSpace(k,\Afpolar)$ are lines.

\begin{cor}\label{cor:penc->ua}
  The class $\tops$ is definable in the structure $\PencSpace(k,\Afpolar)$.
  Consequently, the structure $\PencSpaced(k,\Afpolar)$ and the relation 
  $\upadjac$ on $\sub_k(\Afpolar)$ are definable in $\PencSpace(k,\Afpolar)$.
\end{cor}

Further reasoning to prove \ref{thm:main}
is standard and we will give only a brief overview here.
Let us begin with $k<m-1$. We will show
the induction step, that is, we
start with one of the adjacencies $\downadjac$, $\upadjac$ or $\badjac$ on
$\sub_k(\Afpolar)$ and in terms of such a system interpret adequate
adjacency on $\sub_{k-1}(\Afpolar)$. So, consider two maps:
\begin{itemize}\def\labelitemi{}\itemsep-2pt
\item
$f\colon \sub_{k-1}(\Afpolar)\ni C \longmapsto \starof(C)\in\stars$,
\item
$g\colon \sub_{k-1}(\Afpolar)\ni C \longmapsto \left\{ [C,Y]\colon C\subset Y\in\sub_m(\Afpolar) \right\}
\subset \starsm$.
\end{itemize}

By \ref{prop:afkliki}, $\stars\cup\topsx$ consists of the maximal $\botadjac$-cliques, by 
\ref{lem:difcliq-da-ba}, $\stars$ and $\topsx$ are distinguishable in terms of $\botadjac$, 
and thus the image $\stars$ of $f$ is definable in terms of $\downadjac$ on $\sub_k(\Afpolar)$.
The map $f$ sets a one-to-one correspondence between elements of $\stars$ and
$\sub_{k-1}(\Afpolar)$. Moreover, $\starof(C_1)\cap\starof(C_2)\neq\emptyset$
iff $C_1\upadjac C_2$ which gives $\upadjac$ on $\sub_{k-1}(\Afpolar)$.

Similarly, by \ref{prop:afkliki}, $\starsm\cup\starsx \cup \tops$ is the class of the maximal
$\upadjac$-cliques, and by \ref{cor:difcliq-ua}, $\starsm$ is distinguishable in terms
of $\upadjac$ on $\sub_k(\Afpolar)$.
Finally, by \ref{prop:afkliki}, the elements of $\starsm\cup\topsx$ are the maximal $\badjac$-cliques,
and by \ref{lem:difcliq-da-ba}, $\starsm$ and $\topsx$ are distinguishable in terms of $\badjac$. 
So, the class $\starsm$ is definable on $\sub_k(\Afpolar)$ both in terms of $\upadjac$ 
and in terms of $\badjac$.
Two stars $\K_1, \K_2\in\starsm$ are said to be related iff $|\K_1\cap\K_2|\geq 2$. 
If so, we write $\K_1\approx\K_2$.
If $\K_i = [C_i,Y_i]\in\starsm$, $i=1,2$ and $\K_1\approx \K_2$ then, clearly, $C_1 = C_2$.
Let $C\in\sub_{k-1}(\Afpolar)$. Then $\overline{C}\in\sub_{k-1}(\fixpol)$.
It is known that $[\overline{C},\Quadr]_k$ induces a polar space which is
connected and therefore the 
transitive closure of the relation $\approx$ partitions the family $\starsm$ into equivalence 
classes which uniquely correspond to the elements of $\sub_{k-1}(\Afpolar)$ via the map $g$.
The same trick as in the previous paragraph gives us $\upadjac$ on $\sub_{k-1}(\Afpolar)$.

Note that for $C\in\sub_{k-1}(\Afpolar)$ and $A\in\sub_k(\Afpolar)$
we have $C\subset A$ iff $A\in f(C)$ as well as iff $A\in\bigcup g(C)$. 
Hence, what we have actually defined is $\GrasSpace(k-1,\Afpolar)$.
In turn, the relation $\upadjac$ on $\sub_{k-1}(\Afpolar)$
remains definable in $\GrasSpace(k-1,\Afpolar)$ and 
we can continue our inductive procedure as long as $k\geq 1$ (so, $k-1\geq 0$).
Proceeding inductively, we end up with $\GrasSpace(0,\Afpolar)$ which is,
up to an isomorphism, our affine polar space $\Afpolar$ 
and that way Theorem~\ref{thm:main} is proved.
Note that in case $k<m-1$ Corollary~\ref{cor:main} is an immediate consequence
of the above result.

Now, let us pay attention to the cases $k=m-1$ and $k=m$.
Here, we need some other techniques.
In view of \ref{prop:afkliki} and \ref{fct:geoonkliks}
the maximal cliques of the relation $\upadjac$ defined on $\sub_{m-1}(\Afpolar)$
are the elements of $\tops$, and the maximal cliques of
 $\botadjac$ defined on
$\sub_m(\Afpolar)$ are the elements of $\stars$,
which yields that
\begin{multline} \label{obs:1}
  \text{the structure }
  \GrasSpace(m-1,\Afpolar) 
  \text{ is definable in  both } \struct{\sub_m(\Afpolar),\botadjac}
  \text{ and } 
  \struct{\sub_{m-1}(\Afpolar),\upadjac},
  \text{ and} 
  \\  \text{the structures}
  \struct{\sub_m(\Afpolar),\botadjac} \text{ and }
   \struct{\sub_{m-1}(\Afpolar),\upadjac}
  \text{ are mutually definable}
  \end{multline}
(cf. a particular case of \eqref{obs:1} in \ref{cor:ua->pencp}).

The case where $\badjac$ is defined on $\sub_{m-1}(\Afpolar)$ requires a
different treatment. In this case the function $g$ makes sense as previously, but
now {\em related stars} would coincide, as the elements of $\starsx$ are simply
the lines of $\PencSpace(m-1,\Afpolar)$.  So, it is impossible to identify the
elements of $\sub_{m-2}(\Afpolar)$ with the classes of mutually related stars.

\begin{prop}
  The relation $\upadjac$ on\/ $\sub_{m-1}(\Afpolar)$ can be characterized in terms of
  $\badjac$ defined on $\sub_{m-1}(\Afpolar)$.
  Consequently, $\GrasSpace(m-1,\Afpolar)$ can be recovered within $\PencSpace(m-1,\Afpolar)$.
\end{prop}
\begin{proof}
  In view of \ref{prop:afkliki},
  the maximal cliques of the relation $\badjac$ on $\sub_{m-1}(\Afpolar)$
  are the selectors in $\topsx$ and the stars in $\starsx$.
  From \ref{lem:difcliq-da-ba}, the elements of $\topsx$ and $\starsx$ are distinguishable.
  \par
  For $A_1,A_2,A_3\in\sub_{m-1}$ write $\Pi(A_1,A_2,A_3)$ when $A_1,A_2,A_3\in{\cal K}$
  for some ${\cal K}\in\topsx$ and there is no ${\cal K}'\in\starsx$ with $A_1,A_2,A_3\in{\cal K}'$.
  Directly in terms of $\PencSpace(m-1,\Afpolar)$ one can express this definition as follows:
  $\Pi(A_1,A_2,A_3)$ iff $A_1,A_2,A_3$ are the vertices of a proper triangle. In any case,
  if  $\Pi(A_1,A_2,A_3)$, then $A_1,A_2,A_3\in\topof(M)$ for a uniquely determined 
  $M\in\sub_m(\Afpolar)$.
  To close the proof, we note that the following holds
  \begin{multline}
    B_1 \upadjac B_2 \iff \bigl(\exists A_1,A_2,A_3\in\sub_{m-1}(\Afpolar)\bigr)
	(\exists \K_1,\K_2\in{\topsx})
        \bigl[\;
          \Pi(A_1,A_2,A_3) \Land
	  \\
	  A_1,A_2,A_3,B_1\in\K_1 \Land A_1,A_2,A_3,B_2\in\K_2\;
	\bigr]
  \end{multline}
  for any $B_1,B_2\in\sub_{m-1}(\Afpolar)$.
\end{proof}

Let us consider the following relation 
$\afbotadjacx$
defined for any 
$A_1,A_2 \in \sub_{m-1}(\Afpolar)$:
\begin{multline}\label{def:afbotadjacx}
  A_1 \afbotadjacx A_2 \iff \bigl(\forall D_1,D_2 \in \sub_m(\Afpolar)\bigr)
    \big[\; \wedge_{i=1}^2 (A_i \subset D_i) \implies 
      \\
      D_1 = D_2 \Lor 
      D_1 \downadjac D_2 \Lor 
      \bigl(\exists D\in\sub_m(\Afpolar)\bigr)[\; D \downadjac D_1,D_2\; ]\;
    \big].
\end{multline}

\begin{lem}\label{lem:gras->polar}
  Let $A_1,A_2 \in \sub_{m-1}(\Afpolar)$. The following conditions are 
  equivalent.
  \begin{sentences}\itemsep-2pt
  \item $A_1 \afbotadjacx A_2$.
  \item
    $\overline{A_1} \afbotadjac \overline{A_2}$ holds i.e.
    either $A_1 \downadjac A_2$ or $A_1^\infty = A_2^\infty$
    (i.e. $A_1 \parallel A_2$).
  \end{sentences}
\end{lem}
\begin{proof}
  Let $\overline{A_1} \afbotadjac \overline{A_2}$ hold, so 
  $\dim(\overline{A_1}\cap\overline{A_2}) \geq m-2$. 
  Take $D_1,D_2$ as in \eqref{def:afbotadjacx}, so 
  $\dim(D_1\cap D_2) \geq m-2$.
  If $\dim(D_1\cap D_2) = m-1$ we are through; assume that
  $\dim(D_1\cap D_2) = m-2$.
  From the properties of 
  the polar space $\fixpol$ there is a required subspace $D$
  and thus $A_1 \afbotadjacx A_2$ holds.  

  Now, let $A_1 \afbotadjacx A_2$ and 
  $\dim(\overline{A_1}\cap\overline{A_2}) < m-2$. 
  There are $D'_1,D'_2 \in\sub_m(\fixpol)$ with 
  $D'_1 \cap D'_2 = \overline{A_1}\cap\overline{A_2}$
  and $A_i\subset D'_i$ for $i = 1,2$.
  Note that $D'_1,D'_2 \not\subset\horyzont$ as otherwise there would be
  no $A_1, A_2$.  
  So, we have $D_i:=D_i'\setminus\horyzont$, $i=1,2$ that do not satisfy
  \eqref{def:afbotadjacx} 
  and thus $A_1 \afbotadjacx A_2$ is false.
\end{proof}

Clearly, the $\afbotadjacx$-cliques are restrictions of $\afbotadjac$-cliques
defined on $\sub_{m-1}(\fixpol)$.
In view of \ref{fct:polkliki}, there are three classes of the maximal cliques of 
$\afbotadjacx$: the class $\tops$, the class $\stars$, and the class
\begin{ctext}
  $\starsw = \big\{ \{ U\colon U\parallel U_0 \}\colon U_0\in\sub_{m-1}(\Afpolar) \big\}$.
\end{ctext}
So, in view \ref{lem:gras->polar}, the class $\stars \cup \starsw$ is 
definable in terms of $\GrasSpace(m-1,\Afpolar)$.
\begin{prop}\label{lem:gras->paral}
  The structure $\PencSpaced(m-1,\Afpolar)$ can be defined in terms of\/
  $\GrasSpace(m-1,\Afpolar)$.
  \par
  Thus $\PencSpace(m-1,\Afpolar)$ remains definable in $\GrasSpace(m-1,\Afpolar)$
  and, consequently, one can distinguish $\stars$ and $\starsw$ in terms of the geometry
  of\/ $\GrasSpace(m-1,\Afpolar)$.
  \par
  In particular, $\botadjac$ is definable on $\sub_{m-1}(\Afpolar)$ in terms of $\upadjac$
  and in terms of $\badjac$.
\end{prop}
\begin{proof}
  To justify the first statement it suffices to note that the lines of 
  $\PencSpaced(m-1,\Afpolar)$ are the sets of the form 
  ${\cal K}_1\cap {\cal K}_2$, where ${\cal K}_2\in\tops$ and
  ${\cal K}_1\in\stars\cup\starsw$.
\par
  The second claim is a direct consequence of \ref{cor:pencp->penc}.
  Finally, for $X\in\stars\cup\starsw$ we have
  \begin{ctext}
    $X\in\stars$ iff each two $A_1,A_2 \in X$, if joinable in $\PencSpaced(m-1,\Afpolar)$ 
    lie on a line
	of $\PencSpace(m-1,\Afpolar)$.
  \end{ctext}
  This enables us to distinguish respective types of cliques. 
\end{proof}

Evidently, if $A_1,A_2\in\sub_{m-1}(\Afpolar)$ then $A_1\downadjac A_2$
iff $A_1,A_2 \in {\cal K}$ for some ${\cal K}\in\stars$.
Summing up 
\eqref{obs:1}, \ref{cor:penc->ua},  \ref{lem:gras->paral},
\ref{thm:main}, \ref{lem:gras->polar}, \ref{cor:pencp->penc}, and \ref{thm:main}
we get 
\ref{cor:main}
for $k=m-1$.

\section{Final remarks}

Our main result, Theorem~\ref{thm:main}, reads (formally) more or less the same 
way as the main statements of \cite{afpolar}: Theorem~4.1 together with
Theorem~5.6.
The three adjacencies $\upadjac$, $\downadjac$, $\badjac$ on strong subspaces were
also introduced in \cite{afpolar}.
The key characterization of adjacency cliques in \ref{prop:afkliki} resembles analogous
characterizations with similar formulas in \cite{afpolar}: Fact~2.3 and Proposition~2.4.
Another set of important properties of cliques, gathered in Fact~\ref{fct:geoonkliks},
corresponds to considerations on pages 45-46 in \cite{afpolar}. Definability of pencils
in terms of $\upadjac$-adjacency in Lemma~\ref{lem:linesx} resembles Proposition~3.2 
in \cite{afpolar}. The same idea to distinguish selectors from star cliques
shown in Lemma~\ref{lem:difcliq-da-ba} is also used in Corollary~3.8 in \cite{afpolar}.
An analogue of Corollary~\ref{cor:da+ba->penc} where definability of $\PencSpace(k,\AfPolar)$
in terms of $\badjac$ is stated, has been proved in Proposition~3.12 in \cite{afpolar}.
Finally, definition \eqref{def:afbotadjacx} and use of the relation $\afbotadjacx$
are based on similar idea as 
those used in the proof of Proposition~5.5 in \cite{afpolar}.


\bigskip
\begin{small}
\noindent
Authors' address:
\\
Krzysztof Pra{\.z}mowski,
Mariusz {\.Z}ynel
\\
Institute of Mathematics, University of Bia{\l}ystok
\\
ul. Akademicka 2, 15-267 Bia{\l}ystok, Poland
\\
\verb+krzypraz@math.uwb.edu.pl+,
\verb+mariusz@math.uwb.edu.pl+
\end{small}

\end{document}